\documentclass[a4paper,11pt]{article}
\usepackage[english]{babel}
\usepackage[utf8]{inputenc}
\usepackage{paralist,url,verbatim,anysize}
\usepackage{amscd}
\usepackage{mathrsfs}
\usepackage{microtype}
\usepackage{manfnt}
\usepackage{nicefrac}

\usepackage[nodisplayskipstretch]{setspace}
\usepackage{fancyhdr}
\fancypagestyle{plain}{%
\fancyhf{} 
\fancyfoot[C]{\fontsize{12pt}{12pt}\selectfont\thepage} 

}
\usepackage[e]{esvect}
\usepackage{amsmath,amsthm,amssymb,amsfonts}
\usepackage{graphicx, graphics}
\usepackage[bookmarksopen=true]{hyperref}
\usepackage[font=small,labelfont=bf]{caption}
\usepackage{subcaption}
\theoremstyle{plain}
\newtheorem{theorem}{\bf Theorem}[section]

\newtheorem{cor}[theorem]{Corollary}
\newtheorem{lemma}[theorem]{Lemma}
\newtheorem{proposition}[theorem]{Proposition}

\newtheorem{thmnonumber}{\bf Theorem}

\theoremstyle{definition}

\DeclareMathAlphabet{\mathpzc}{OT1}{pzc}{m}{it}
\newenvironment{frcseries}{\fontfamily{frc}\selectfont}{}

\newcommand{\TT}{\mathrm{T}}

\newcommand{\rint}{\mathrm{relint}}
\newcommand{\RN}{\mathrm{N}}
\newcommand{\e}{\varepsilon}
\newcommand{\R}{\mathbb{R}}
\newcommand{\RS}{\mathrm{R}\,}

\newcommand{\Lk}{\mathrm{Lk}\, }

\newcommand{\St}{\mathrm{St}\, }

\newcommand{\sd}{\mathrm{sd}\, }

\newcommand{\cm}[1]{}

\begin{document}

\author{
Karim Adiprasito\\
\small Einstein Institute for Mathematics\\ \small Hebrew University of Jerusalem\\
\small 91904 Jerusalem, Israel\\
\small \url{adiprasito@math.huji.ac.il}
\and
\and 
Bruno Benedetti \\
\small Department of Mathematics\\ \small University of Miami\\
\small Coral Gables, FL 33146, USA\\
\small \url{bruno@math.miami.edu}}

\date{\small May 9, 2018}
\title{A Cheeger-type exponential bound for the number of triangulated manifolds}
\maketitle
\bfseries

\mdseries


\begin{abstract}
In terms of the number of triangles, it is known that there are more than exponentially many triangulations of surfaces, but only exponentially many triangulations of surfaces with bounded genus. In this paper we provide a first geometric extension of this result to higher dimensions. We show that in terms of the number of facets, there are only exponentially many geometric triangulations of space forms with bounded geometry in the sense of Cheeger (curvature and volume bounded below, and diameter bounded above). This establishes a combinatorial version of Cheeger's finiteness theorem. 

\noindent Further consequences of our work are:

(1) There are exponentially many geometric triangulations of $S^d$. 

(2) There are exponentially many convex triangulations of the $d$-ball.
\end{abstract}

\section{Introduction}
In discrete quantum gravity, one simulates Riemannian structures by considering all possible triangulations of manifolds~\cite{Regge, ADJ, Weingarten}. The metric is introduced a posteriori, by assigning to each edge a certain length (as long as all triangular inequalities are satisfied). For example, in Weingarten's \emph{dynamical triangulations} model,  we simply assign to all edges length $1$, and view all triangles as equilateral triangles in the plane~\cite{ADJ, Weingarten}. The resulting intrinsic metric is  sometimes called ``equilateral flat metric'', cf.~\cite{AB-part1}.

This model gained popularity due to its simplification power. For example, the \emph{partition function} for quantum gravity, a path integral over all Riemannian metrics, becomes a sum over all possible triangulations with $N$ facets \cite{Weingarten}. To make sure that this sum converges when $N$ tends to infinity, one needs to establish an exponential bound for the number of triangulated $d$-manifolds with $N$ facets; compare Durhuus--Jonsson~\cite{DJ}. However, already for $d=2$, this dream is simply impossible: It is known that here are more than exponentially many surfaces with $N$ triangles.  
For $d=2$ the problem can be bypassed by restricting the topology, because for fixed $g$ there are only exponentially many triangulations of the genus-$g$ surface, as explained in~\cite{ADJ, Tutte}. 

In dimension greater than two, however, it is not clear which geometric tools to use to provide exponential cutoffs for the class of triangulations with $N$ simplices.  Are there only exponentially many triangulations of $S^3$, or more? This open problem, first asked in \cite{ADJ1}, was later put into the spotlight also by Gromov~\cite[pp.~156--157]{GromovQuestion}.
Part of the difficulty is that when $d \ge 3$ many $d$-spheres cannot be realized as boundaries of $(d+1)$-polytopes \cite{Kalai, PZ}, and cannot even be shelled \cite{HZ}. In fact, we know that shellable spheres are only exponentially many~\cite{BZ}.

We tackle the problem from a new perspective.  Cheeger's finiteness theorem states that there are only finitely many diffeomorphism types of space forms with ``bounded geometry'': curvature and volume bounded below, and diameter bounded above. What we achieve is a discrete analogue of Cheeger's theorem, which (roughly speaking) shows that \emph{geometric triangulations} of manifolds with bounded geometry are very few. 

\begin{thmnonumber}[Theorem~\ref{thm:DiscreteCheeger}]\label{mainthm:cheeger} 
In terms of the number of facets, there are exponentially many geometric triangulations of space forms with bounded geometry (and fixed dimension).
\end{thmnonumber}
 
Since every topological triangulation of an orientable surface can be straightened to a geometric one \cite{Verdiere, Wagner}, this result is a generalization of the classical exponential bound on the number of triangulated surfaces with bounded genus.

Here is the proof idea. Via Cheeger's bounds on the injectivity radius, we chop any manifold of constant curvature into a finite number of convex pieces of small diameter. Up to performing a couple of barycentric subdivisions, we can assume that each piece is a shellable ball \cite{AB-Combinatorica}, and in particular endo-collapsible \cite{Benedetti-DMT4MWB}. This implies an upper bound for the number of critical faces that a discrete Morse function on the triangulation can have. From here we are able to conclude, using the second author's result that there are only exponentially many triangulations of manifolds with bounded discrete Morse vector \cite{Benedetti-DMT4MWB}.

Inspired by Gromov's question, let us now consider the unit sphere $S^d$ with its standard intrinsic metric. Let us agree to call ``geometric triangulation'' any tiling of $S^d$ into regions that are convex simplices with respect to the given metric and combinatorially form a simplicial complex. For example, all the boundaries of $(d+1)$-polytopes yield geometric triangulations of $S^d$, but not all geometric triangulations arise this way. How many geometric triangulations are there?

Once again, by proving that the second derived subdivision of every geometric triangulation is endo-collapsible, we obtain

\begin{thmnonumber}[Theorem~\ref{thm:stbound}]\label{mainthm:Gromov} 
There are at most $2^{d^2\cdot ((d+1)!)^2 \cdot N}$ distinct combinatorial types of geometric triangulations of the standard~$S^d$ with $N$ facets.
\end{thmnonumber}

Our methods rely, as explained, on convex and metric geometry. Whether \emph{all} triangulations of $S^d$ are exponentially many, remains open. But even if the answer turned out to be negative, Main Theorems~\ref{mainthm:cheeger} and~\ref{mainthm:Gromov} provide some support for the hope of discretizing quantum gravity in all dimensions.

\section*{Preliminaries}

By $\R^d$, $\mathbb{H}^d$ and $S^d$ we denote the euclidean $d$-space, the hyperbolic $d$-space, and the unit sphere in $\R^{d+1}$, respectively. A \emph{(euclidean) polytope} in $\R^d$ is the convex hull of finitely many points in $\R^d$. Similarly, a \emph{hyperbolic polytope} in $\mathbb{H}^d$ is the convex hull of finitely many points of $\mathbb{H}^d$. A \emph{spherical polytope} in $S^d$ is the convex hull of a finite number of points that all belong to some open hemisphere of $S^d$. Spherical polytopes are in natural one-to-one correspondence with euclidean polytopes, just by taking radial projections; the same is true for hyperbolic polytopes. A \emph{geometric polytopal complex} in $\R^d$ (resp.\ in $S^d$ or $\mathbb{H}^d$) is a finite collection of polytopes in $\R^d$ (resp.~$S^d$, $\mathbb{H}^d$) such that the intersection of any two polytopes is a face of both. An \emph{intrinsic polytopal complex} is a collection of polytopes that are attached along isometries of their faces (cf.\ Davis--Moussong~\cite[Sec.\ 2]{DavisMoussong}), so that the intersection of any two polytopes is a face of both. 

The \emph{face poset} $(C, \subseteq)$ of a polytopal complex $C$ is the set of nonempty faces of $C$, ordered with respect to inclusion.  Two polytopal complexes $C,\, D$ are \emph{combinatorially equivalent}, denoted by $C\cong D$, if their face posets are isomorphic. Any polytope combinatorially equivalent to the $d$-simplex, or to the regular unit cube $[0,1]^d$, shall simply be called a \emph{$d$-simplex} or a \emph{$d$-cube}, respectively. A polytopal complex is \emph{simplicial}  (resp.~\emph{cubical}) if all its faces are simplices (resp.~cubes). 

The \emph{underlying space} $|C|$ of a polytopal complex $C$ is the topological space obtained by taking the union of its faces. If two complexes are combinatorially equivalent, their underlying spaces are homeomorphic. 
If $C$ is simplicial, $C$ is sometimes called a \emph{triangulation} of $|C|$ (and of any topological space homeomorphic to $|C|$). If $X$ is a metric length space, we call \emph{geometric triangulation} any intrinsic simplicial complex $C$ isometric to $X$ such that the simplices of $C$ are convex in the underlying length metric. 
For example, the boundary of every simplicial $(d+1)$-polytope yields a geometric triangulation of~$S^{d}$. 

\begin{figure}[htb]
	\centering
  \includegraphics[width=.19\linewidth]{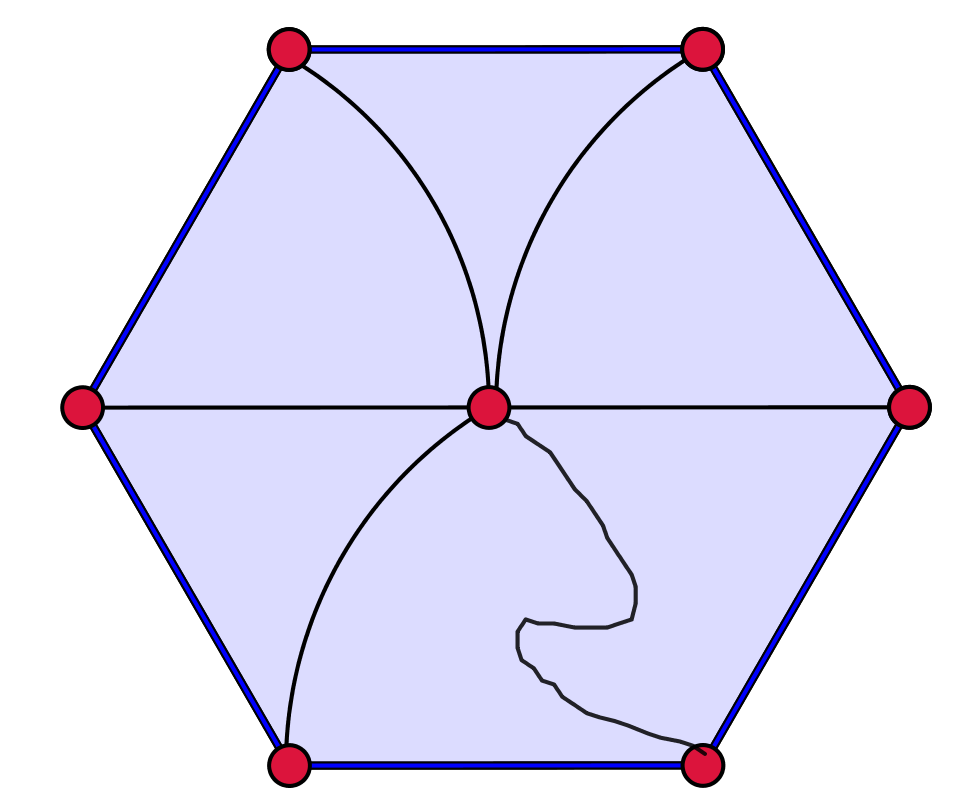}
  \hskip10mm
    \includegraphics[width=.21\linewidth]{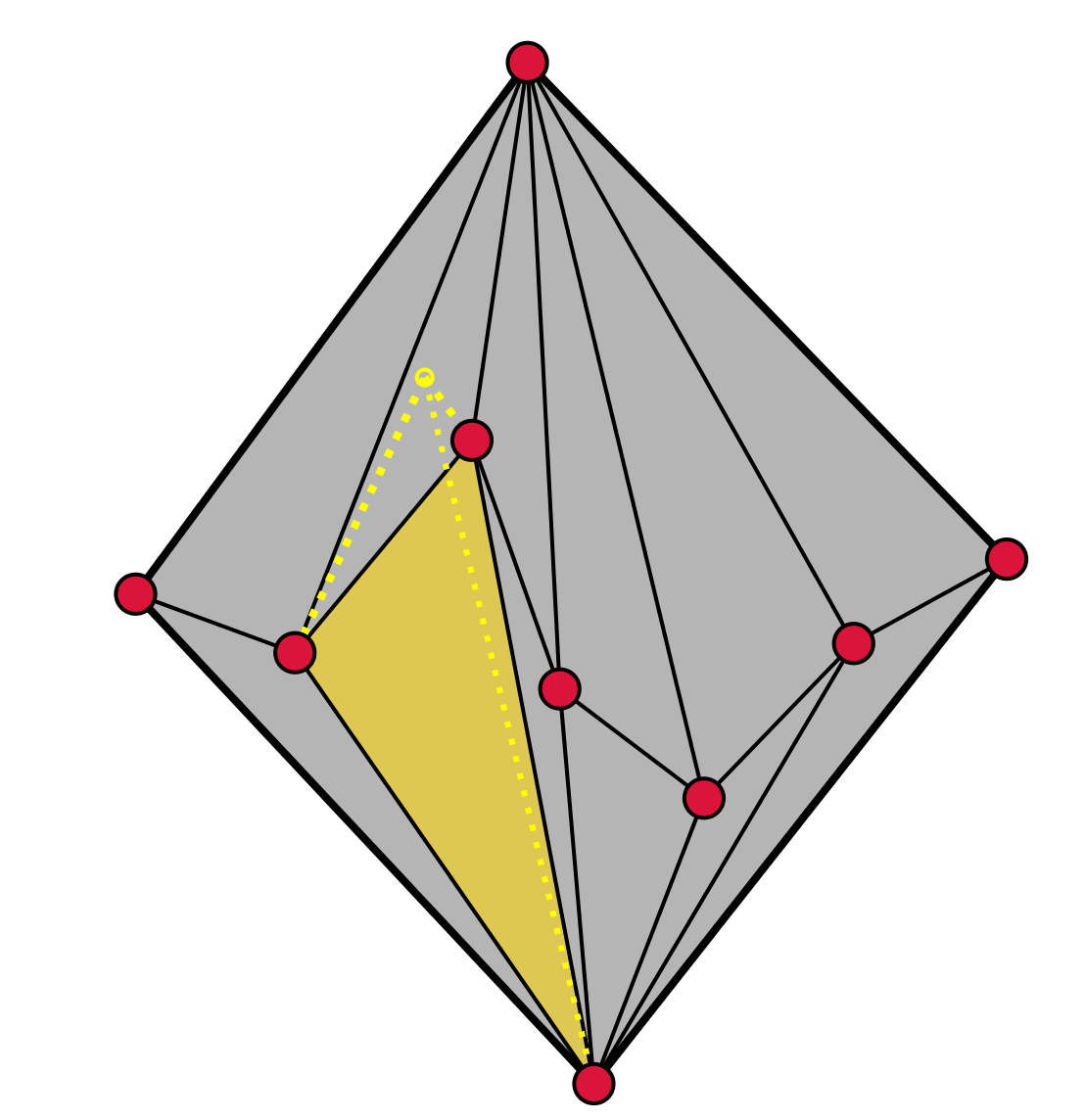}
 	\caption{\footnotesize \textsc{Left:} A tiling of a disk that is not a geometric triangulation, because the tiles are not convex. \textsc{Right:} Rudin's $3$-ball $R$ is a non-shellable subdivision of a convex $3$-dimensional polytope with $14$ vertices, cf.~\cite{Wotzlaw}. Coning off the boundary of $R$ one gets a simplicial complex $\partial (v \ast R)$ that is a geometric triangulation of $S^3$, but it is not shellable, hence not polytopal.}
	\label{fig:NonGeo}
\end{figure}

A \emph{subdivision} of a polytopal complex $C$ is a polytopal complex $C'$ with the same underlying space of $C$, such that for every face $F'$ of $C'$ there is some face $F$ of $C$ for which $F' \subset F$. A \emph{derived subdivision} $\sd  C$ of a polytopal complex $C$ is any subdivision of $C$ obtained by stellarly subdividing at all faces in order of decreasing dimension of the faces of $C$, cf.\ \cite{Hudson}.  An example of a derived subdivision is the \emph{barycentric subdivision}, which uses as vertices the barycenters of all faces of $C$.
 
If $C$ is a polytopal complex, and $A$ is some set, we define the \emph{restriction $\RS(C,A)$ of $C$ to $A$} as the inclusion-maximal subcomplex $D$ of $C$ such that $|D|$ lies in $A$. The \emph{star} of $\sigma$ in $C$, denoted by $\St(\sigma, C)$, is the minimal subcomplex of $C$ that contains all faces of $C$ containing $\sigma$. The \emph{deletion} $C-D$ of a subcomplex $D$ from $C$ is the subcomplex of $C$ given by $\RS(C,  C{\setminus} \rint{D})$. The \emph{(first) derived neighborhood} $N(D,C)$ of $D$ in $C$ is the simplicial complex
\[N(D,C):=\bigcup_{\sigma\in \sd D} \St(\sigma,\sd C). \] 
Similarly, for $k > 1$,  the $k$-th derived subdivision of a complex $C$ is recursively defined as $\sd C = \sd (\sd^{k-1} C)$. The \emph{$k$-derived neighborhood of $D$ in $C$}, denoted by $N^k(D,C)$, is the union of $\St(\sigma, \sd^k C)$, with $\sigma \in \sd^k D$.

\begin{figure}[hbtp] 	\centering
\includegraphics[scale=.44]{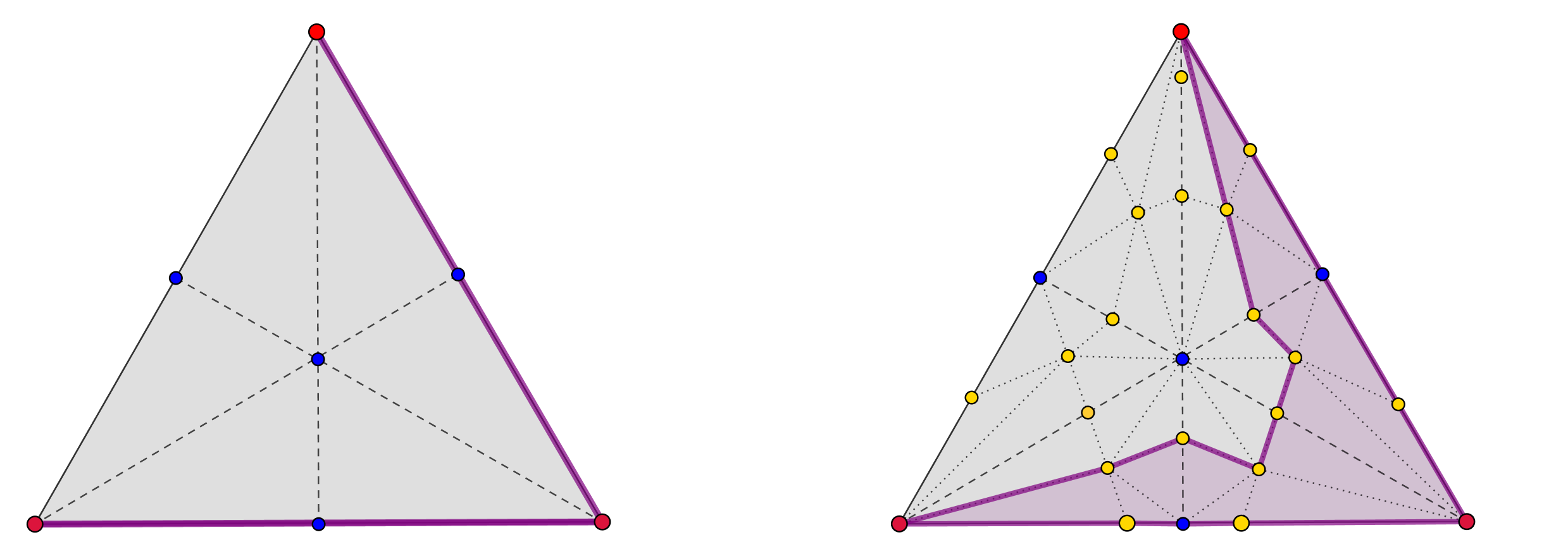} 
\caption{\footnotesize $N^2(D,C)$, where $C$ is a triangle and $D$ the subcomplex formed by two of its edges.}
\end{figure}

Next comes the geometric definition of the link of a face $\sigma$. Intuitively, it is a spherical complex whose face poset is the upper order ideal of $\sigma$ in the face poset of $C$. The formal definition is as follows, cf.~\cite{AB-part2}. Let $p$ be any point of a metric space $X$. By $\TT_p X$ we denote the tangent space of $X$ at $p$. Let $\TT^1_p X$ be the restriction of $\TT_p X$ to unit vectors.  If $Y$ is any subspace of $X$, then $\RN_{(p,Y)} X$ denotes the subspace of the tangent space $\TT_p X$ spanned by the vectors orthogonal to $\TT_p Y$. If $p$ is in the interior of $Y$, we define $\RN^1_{(p,Y)} X:= \RN_{(p,Y)} X \cap \TT^1_p Y$.
If $\tau$ is any face of a polytopal complex $C$ containing a nonempty face $\sigma$ of $C$, then the set $\RN^1_{(p,\sigma)} \tau$ of unit tangent vectors in $\RN^1_{(p,\sigma)} |C|$ pointing towards $\tau$ forms a spherical polytope $P_p(\tau)$, isometrically embedded in $\RN^1_{(p,\sigma)} |C|$. The family of all polytopes $P_p(\tau)$ in $\RN^1_{(p,\sigma)} |C|$ obtained for all $\tau \supset \sigma$ forms a polytopal complex, called the \emph{link} of $C$ at $\sigma$; we will denote it by $\Lk_p(\sigma, C)$. If $C$ is a geometric polytopal complex in $X^d=\R^d$ (or $X^d=S^d$), then $\Lk_p(\sigma, C)$ is naturally realized in $\RN^1_{(p,\sigma)} X^d$. Obviously, $\RN^1_{(p,\sigma)} X^d$ is isometric to a sphere of dimension $d-\dim \sigma -1$, and will be considered as such. Up to ambient isometry $\Lk_p(\sigma, C)$ and  $\RN^1_{(p,\sigma)} \tau$ in $ \RN^1_{(p,\sigma)} |C|$ or $\RN^1_{(p,\sigma)} X^d$ do not depend on $p$; for this reason, $p$ will be omitted in notation whenever possible. By convention, we define $\Lk(\emptyset, C)=C$, and it is the only link that does not come with a natural spherical metric.

If $C$ is a simplicial complex, and $\sigma$, $\tau$ are faces of $C$, we denote by $\sigma\ast \tau$ the minimal face of $C$ containing both $\sigma$ and $\tau$ (if there is one). If $\sigma$ is a face of $C$, and $\tau$ is a face of $\Lk(\sigma,C)$, then $\sigma \ast \tau$ is defined as the face of $C$ with $\Lk(\sigma,\sigma \ast \tau)=\tau$. In both cases, the operation~$\ast$ is called the \emph{join}.

\begin{figure}[htb]
	\centering
  \includegraphics[width=.39\linewidth]{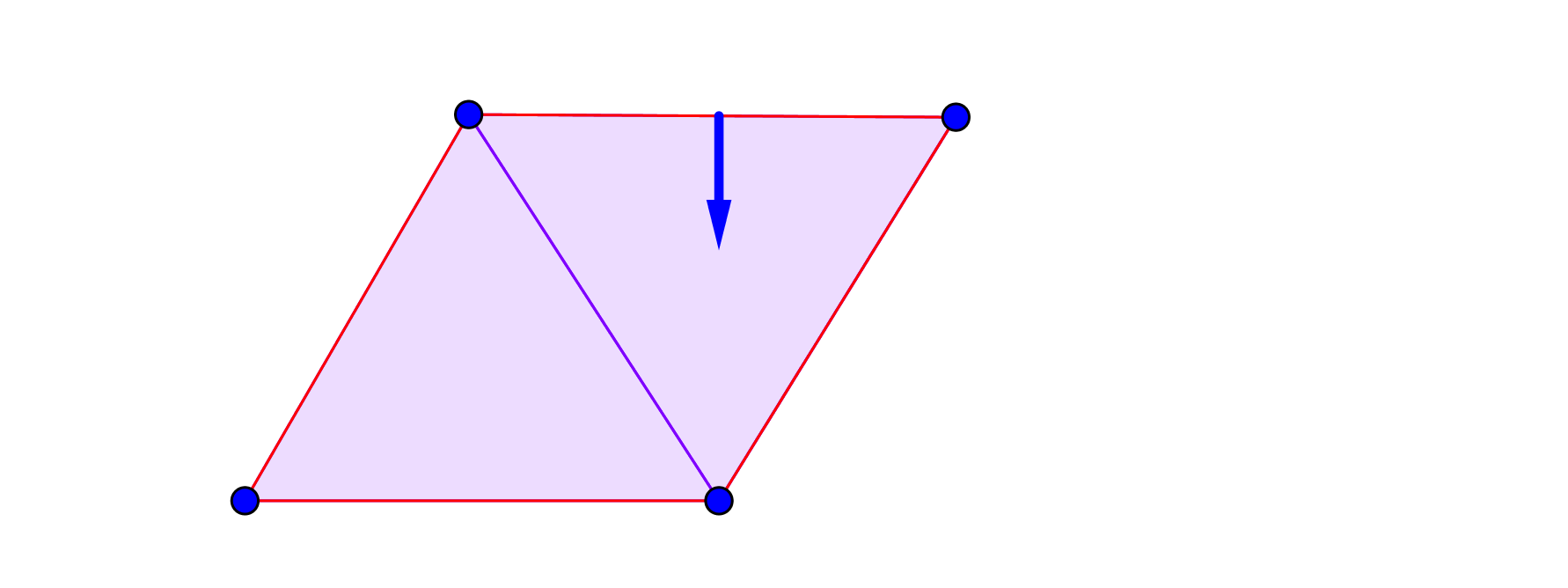}
  \hskip4mm
    \includegraphics[width=.39\linewidth]{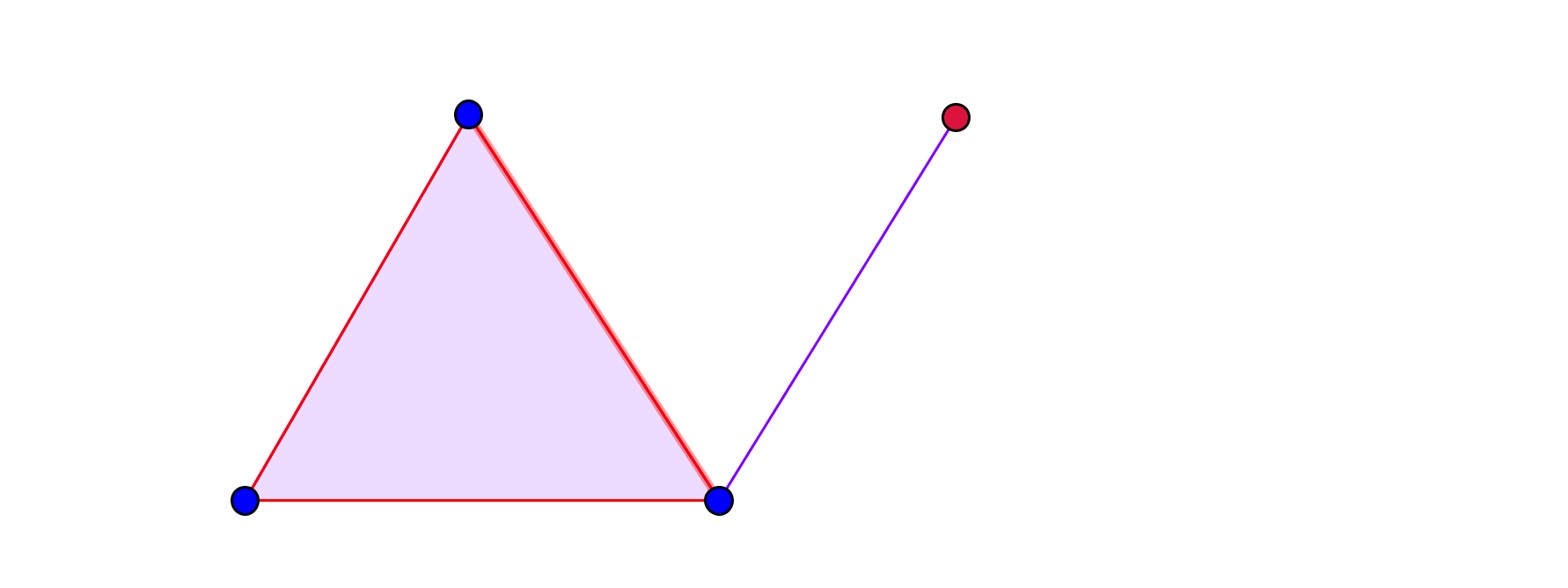}
 	\caption{\footnotesize The complex on the left has four free edges (in red). The deletion of any of them is callled an `elementary collapse'. It yields a complex with two fewer faces and same homotopy type (right). Such complex still has free faces, so the simplification process can be carried out further.}
	\label{fig:ElementaryCollapse}
\end{figure}

Inside a polytopal complex $C$, a \emph{free} face $\sigma$ is a face strictly contained in only one other face of $C$. An \emph{elementary collapse} is the deletion of a free face $\sigma$ from a polytopal complex~$C$. We say that $C$ \emph{(elementarily) collapses} onto $C-\sigma$, and write $C\searrow_e C-\sigma.$ We also say that the complex $C$ \emph{collapses} to a subcomplex $C'$, and write~$C\searrow C'$, if $C$ can be reduced to $C'$ by a sequence of elementary collapses. A \emph{collapsible} complex is a complex that collapses onto a single vertex. Collapsibility is, clearly, a combinatorial property (i.e. it only depends on the combinatorial type), and does not depend on the geometric realization of a polytopal complex. We have however the following results:

\begin{theorem}[Adiprasito--Benedetti \cite{AB-part2}] \label{thm:ConvexEndo}
Let $C$ be a simplicial complex. If the underlying space of $C$ in $\mathbb{R}^d$ is convex, then the (first) derived subdivision of $C$ is collapsible.
\end{theorem}

We are also going to apply to certain face links the following ``spherical version'' of the statement above.

\begin{theorem}[{\cite{AB-part2}, cf.~also \cite[Chapter II.2]{Athesis}}] \label{thm:ConvexEndo2}
Let $C$ be any convex polytopal $d$-complex in $S^d$. Let $\overline{H}_+$ be a closed hemisphere of $S^d$ in general position with respect to $C$. 
\begin{compactenum}[\rm (A)]
\item If $\partial C\cap \overline{H}_+=\emptyset$, then $N(\RS(C,\overline{H}_+),C)$ is collapsible.
\item If $\partial C\cap \overline{H}_+$ is nonempty, and $C$ does not lie in $\overline{H}_+$, then $N(\RS(C,\overline{H}_+),C)$ collapses to the subcomplex $N(\RS(\partial C,\overline{H}_+),\partial C)$.
\item If $C$ lies in $\overline{H}_+$, then there exists some facet $F$ of $\sd \partial C$ such that $\sd C$ collapses to $C_F:=\sd \partial C-F$.
\end{compactenum}
\end{theorem}

\section{Geometric triangulations}
If we want to reach exponential bounds for triangulations of $d$-manifolds, and $d$ is at least two, we \emph{must} add some geometric or topological assumption. In fact, already for $d=2$, it is easy to construct $g!$ combinatorially inequivalent triangulations with $14g+5$ triangles of the genus-$g$ surface, cf.~Figure \ref{fig:3}. Setting $N = 14g+5$, clearly  $\lfloor \frac{N}{14} \rfloor !$ grows faster than any exponential. This motivates the search for an exponential upper bound to the number of triangulations with extra geometric properties. 

\begin{figure}[hbtp]  \centering
\includegraphics[scale=.10]{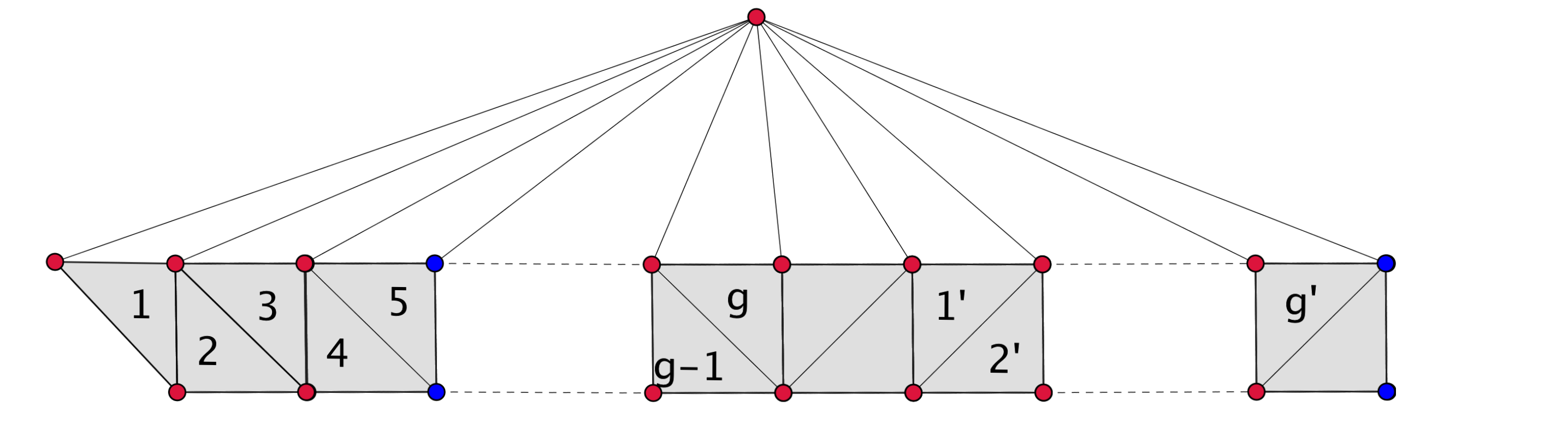} 
\caption{To construct $g!$ triangulations of the genus-$g$ surface, fix a bijection $\pi: \{1, \ldots, g\} \rightarrow \{1, \ldots, g'\}$. Take a strip of $2g+3$ triangles as above, and cone off its boundary to get a $2$-sphere with $4g+5$ triangles. Now remove the interiors of the $2g$ triangles $1, \ldots, g, 1', \ldots, g'$, and create a genus-$g$ surface by attaching handles between the hole $i$ and the hole $\pi(i)$, for all $i$. Since every handle can be triangulated using $12$ triangles, this yields a triangulation $T_{\pi}$  with $14g+5$ triangles of the genus-$g$ surface. By inspecting the link of the highest-degree vertex, from $T_{\pi}$ we can recover $\pi$, which implies that different permutations yield different triangulations.\label{fig:3}}	\end{figure}

Let us recall the notion of {\it endo-collapsibility}, introduced in \cite{Benedetti-DMT4MWB}. A triangulation $C$ of a $d$-manifold with non-empty boundary is called \emph{endo-collapsible} if $C$ minus some $d$-face $\Sigma$ collapses onto $\partial C$. A triangulation $C$ of a $d$-manifold with empty boundary is called \emph{endo-collapsible} if $C$ minus some $d$-face $\Sigma$ collapses onto some vertex $v$. (The choice of $\Sigma$ and that of $v$ do not matter: if a sphere $C$ is endo-collapsible, then for any face $\Delta$ and for any vertex $w$, one has that $C$ minus $\Delta$ collapses onto $w$.)  

\begin{figure}[hbtp] 	\centering
\includegraphics[scale=.39]{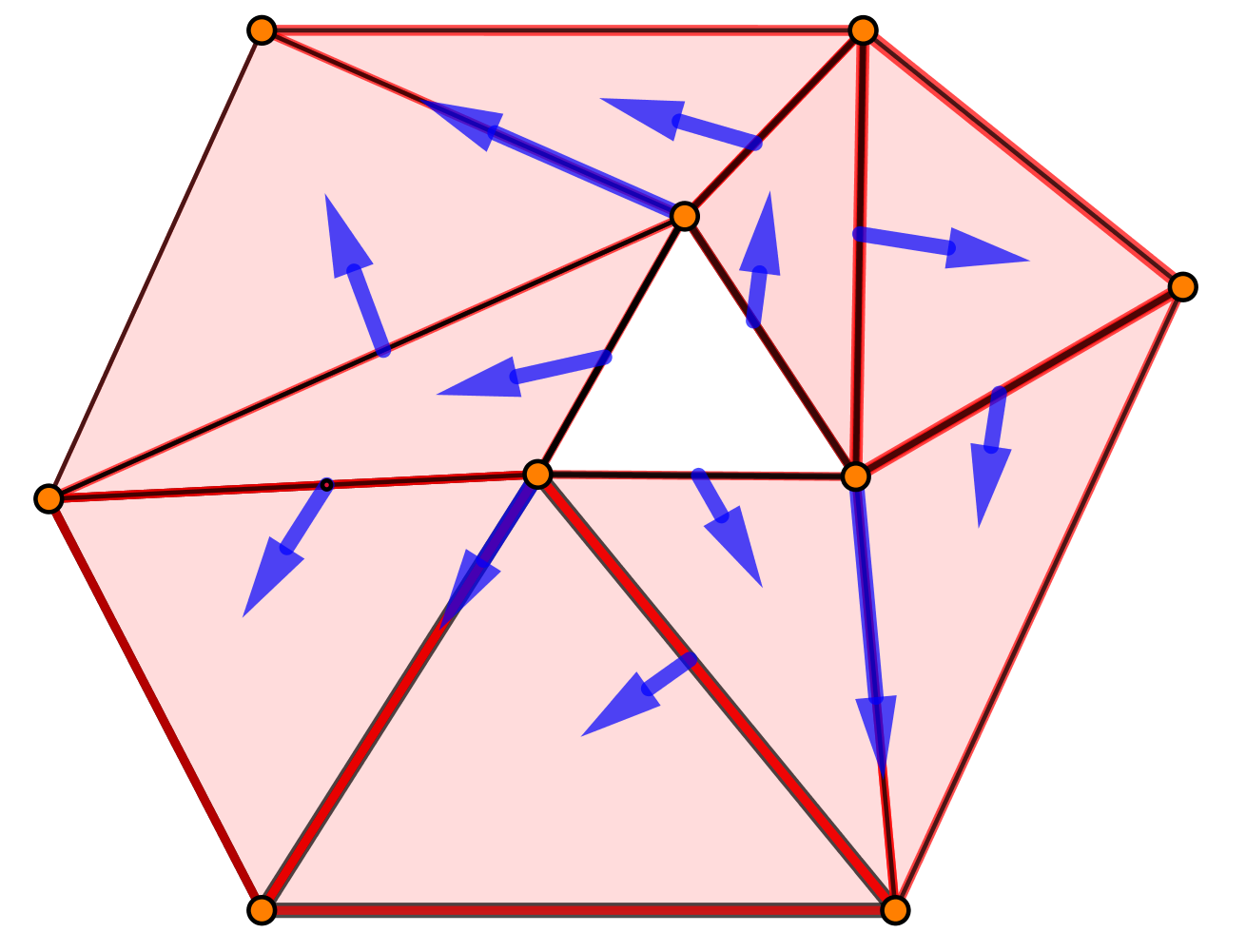} 
\caption{A $d$-ball is endo-collapsible if it collapses to the boundary after removing the interior of some $d$-simplex (it does not matter which). In dimension $2$, all disks are endo-collapsible. We represent an elementary collapse $(\sigma, \Sigma)$ by drawing an arrow between the barycenters of the two faces, with its tail on the lower-dimensional one. 
}
\end{figure}

This notion is of interest to us for the following exponential upper bound:

\begin{theorem}[Benedetti--Ziegler~{\cite{BZ}}]\label{thm:ENDOexpbound}
For fixed $d$, in terms of the number $N$ of facets, there are at most $2^{d^2 \cdot N}$ endo-collapsible triangulations of $d$-manifolds with $N$ facets.
\end{theorem}

Not all triangulations of simplicial balls and spheres are endo-collapsible, cf.\ \cite[Thm.~3]{BZ}. An even stronger fact is that many triangulated spheres are still not endo-collapsible after two barycentric subdivisions \cite{Benedetti-DMT4MWB}. 
However, it turns out that all \emph{geometric} triangulations of spheres become endo-collapsible after (at most) two derived subdivisions:

\begin{lemma} [{Benedetti~\cite[Corollary~3.21]{Benedetti-DMT4MWB}}]
\label{lem:endcollapse}
Let $B$ be a collapsible triangulation of the $d$-ball. If $\sd \Lk (\sigma,B) $ is endo-collapsible for every face $\sigma$, then $\sd B$ is endo-collapsible. 
\end{lemma}

 \begin{proposition} \label{prp:sphereendo}
Let $C$ denote a geometric triangulation of a convex subset of $S^d$. Then $\mathrm{sd}^2 C$ is endo-collapsible.
\end{proposition}

\begin{proof}
Let $H$ denote a general position hemisphere of $S^d$, and let $H^c$ denote the complement. Then by Theorem~\ref{thm:ConvexEndo2}, $N^1(\RS(C,H),C)$ and $N^1(\RS(C,H^c),C)$ are both collapsible. By Lemma~\ref{lem:endcollapse}, $\sd N^1(\RS(C,H),C)$ and $\sd N^1(\RS(C,H^c),C)$ are endo-collapsible. Thus 
$\mathrm{sd}^2 C$ is endo-collapsible. 
\end{proof}

Using the bound from Theorem~\ref{thm:ENDOexpbound}, we can achieve three non-trivial exponential bounds. 

\begin{lemma}[{following Bayer \cite[Theorem 3]{Bayer}}] \label{lem:bayer}
Let $A$ and $B$ be two triangulations of the same manifold. If $\sd A= \sd B$, then $A=B$.
\end{lemma}

\begin{proof} We make the stronger claim that the barycentric subdivision of any \emph{polytopal} complex $T$ determines $T$ up to duality. The claim immediately implies the conclusion, because if two complexes are dual to one another, only one of them can be simplicial. To prove our claim, we proceed by induction on $\dim T$. Without loss of generality, assume $T$ is connected. Also, assume $T$ is not a single simplex. Following Bayer \cite{Bayer}, consider a minimal coloring of the complex $\sd T$. Choose a vertex $v$ of $\sd T$. Decompose $\Lk(v, \sd T)$ minimally into color classes so that $\Lk(v,\sd T)$ is a join of the complexes induced by these color classes. If this decomposition is trivial, $v$ corresponds to either a vertex or a facet of $T$. 
List all the vertices with trivial decomposition. By the assumption, the induced complex on these vertices is a 
2-colorable graph with an edge, with endpoints $a$ and $b$. If $a$ is a 
vertex of $T$, then $b$ must correspond to a facet of $T$. By inductive assumption we can determine $\Lk(a,T)$ from $\Lk(a,\sd T)$; in particular, $T$ is determined by $\sd T$. The other option is if $a$ is a facet; in this case $b$ is a vertex of $T$ and again by induction we can determine $\Lk(b,T)$ from $\Lk(b,\sd T)$. Hence, there are only two options, which 
are dual to one another.
\end{proof}

\begin{theorem}\label{thm:stbound} In terms of the number $N$ of facets, there are: \begin{compactitem}[$\bullet$]
\item less than $2^{d^2\cdot (d+1)! \cdot N}$ geometric triangulations of convex balls in $\R^d$,
\item less than $2^{d^2\cdot ((d+1)!)^2 \cdot N}$ geometric triangulations of $S^d$, 
\item less than $2^{d^2\cdot ((d+1)!)^{(d-2)} \cdot N}$ star-shaped balls in $\R^d$.
\end{compactitem}
\end{theorem}

\begin{proof}
We only prove the first item here, the other ones can be proven analogously. Let $C$ be any simplicial subdivision of a convex $d$-polytope. By Theorem~\ref{thm:ConvexEndo2}(C), the derived subdivision $\sd C$ is endo-collapsible. Furthermore, if $C$ has $N$ facets, then $\sd C$ has $(d+1)!\cdot N$ facets. Hence, by Theorem~\ref{thm:ENDOexpbound}, $\sd C$ is one of at most $2^{d^2\cdot (d+1)! \cdot N}$ combinatorial types. Since  simplicial complexes with isomorphic derived subdivisions are isomorphic by Lemma \ref{lem:bayer}, we conclude that $C$ is one of at most $2^{d^2\cdot (d+1)! \cdot N}$ combinatorial types.
\end{proof}

\section{Triangulated space forms with bounded geometry}
In this section, we wish to study \emph{space forms}, which are Riemannian manifolds of constant sectional curvature \cite{CheegerComp}. We focus on space forms with ``bounded geometry'', with the goal of establishing an exponential upper bound for the number of triangulations. The following Lemma is well known:

\begin{lemma} \label{lem:lowerbound}
Let $M$ be a space form of dimension $d \ge 2$. There are \emph{at least} exponentially many triangulations of $M$. 
\end{lemma}

This motivates the search for an upper bound to the number of such geometric triangulations. We will show that Lemma~\ref{lem:lowerbound} is best possible, in the sense that these triangulations are also \emph{at most} exponentially many (Theorem~\ref{thm:DiscreteCheeger}). 

Our idea is to chop a geometric triangulation of a space form with bounded geometry into a bounded number of endo-collapsible balls. The key for this is given by the following two lemmas: One is Cheeger's bound on the injectivity radius, the other a direct consequence of Toponogov's theorem.

\begin{lemma}[Cheeger~\cite{CheegerCrit}]
Let $-\infty<k<\infty$ and $D,V>0$. There exists a positive number $\widetilde{\mathcal{C}}(k,D,V)>0$ such that every Riemannian $d$-manifold with 
curvature $\ge k$, diameter $\le D$ and volume $\ge V$, 
has no closed geodesic of length less than $\widetilde{\mathcal{C}}(k,D,V)$.
\end{lemma}

A well-known result by Klingenberg is that the injectivity radius is larger than the minimum of half the length of the shortest closed geodesic and $\sqrt{K}$, where $K$ is the supremum of sectional curvatures on the Riemannian manifold in question. From this we conclude:

\begin{cor}
Let $-\infty<k<\infty$ and $D,V>0$. There exists an integer $\mathcal{C}(k,D,V)>0$ such that every $d$-dimensional space form with 
curvature $\ge k$, diameter $\le D$ and volume $\ge V$ has injectivity radius at least $\mathcal{C}(k,D,V)$.
\end{cor}

Finally, we need a lemma to cover a Riemannian manifold by disks.

\begin{lemma}\label{lem:finitecover}
Let $k, D$ be real numbers, with $D>0$. Let $d$ be a positive integer. For every $\e>0$, there exists a positive integer $N_\e$ such that every Riemannian $d$-manifold with curvature bounded below by $k$ and diameter at most $D$ can be covered with at most $N_\e$ balls of radius $\e$.
\end{lemma}

\begin{proof}
Let $X$ be a Riemannian manifold satisfying the assumptions. Let $x$ be a point of $X$. 
Let $B^d_D$ be a ball of radius $D$ in the $d$-dimensional space 
of constant curvature $m = \min \{ k,0 \}$. The ball $B^d_D$ has a cover with 
$N_\e$ balls of diameter $\e$.
Consider the map $ \exp_x \widehat{\exp}^{-1}$, where $\widehat{\exp}$ is 
the exponential map in $B^d_D$ with respect to the center and $\exp_x$ 
is the exponential map in $X$ with respect to $x$, cf.\ \cite[Ch.~1,~Sec.~2]{CheegerComp}.
By the assumption, $X$ has curvature bounded below by $k$. By Toponogov's Theorem, for all $a,b\subset B^d_D$, we have \[|(\exp_x \widehat{\exp}^{-1})(a) (\exp_x\widehat{\exp}^{-1})(b)|\leq |ab|;\] in other words, $\exp_x \widehat{\exp}^{-1}$ is a non-expansive map.
Thus, the images of the $N_\e$ balls that cover $B^d_D$ are contained 
in $N_\e$ balls of radius at most $\e$.
\end{proof}

We are ready for the proof of the main theorem.

\begin{theorem} \label{thm:DiscreteCheeger}
For fixed $k, D, d$ and $V$, in terms of the number $N$ of facets, there are only exponentially many intrinsic simplicial complexes whose underlying spaces are $d$-dimensional space forms of curvature bounded below by $k$, of diameter $\le D$ and of volume $\ge V$. 
\end{theorem}

\begin{proof}
Our proof has three parts:
\begin{compactenum}[\bf (I)]
\item we cover a space form $X$ satisfying the constraints above with convex open balls;
\item we count the number of geometric triangulations restricted to each ball; 
\item we assemble the triangulated balls together, thus estimating the number of triangulations of $X$.
\end{compactenum}

\begin{figure}[htb]
	\centering
  \includegraphics[width=.34\linewidth]{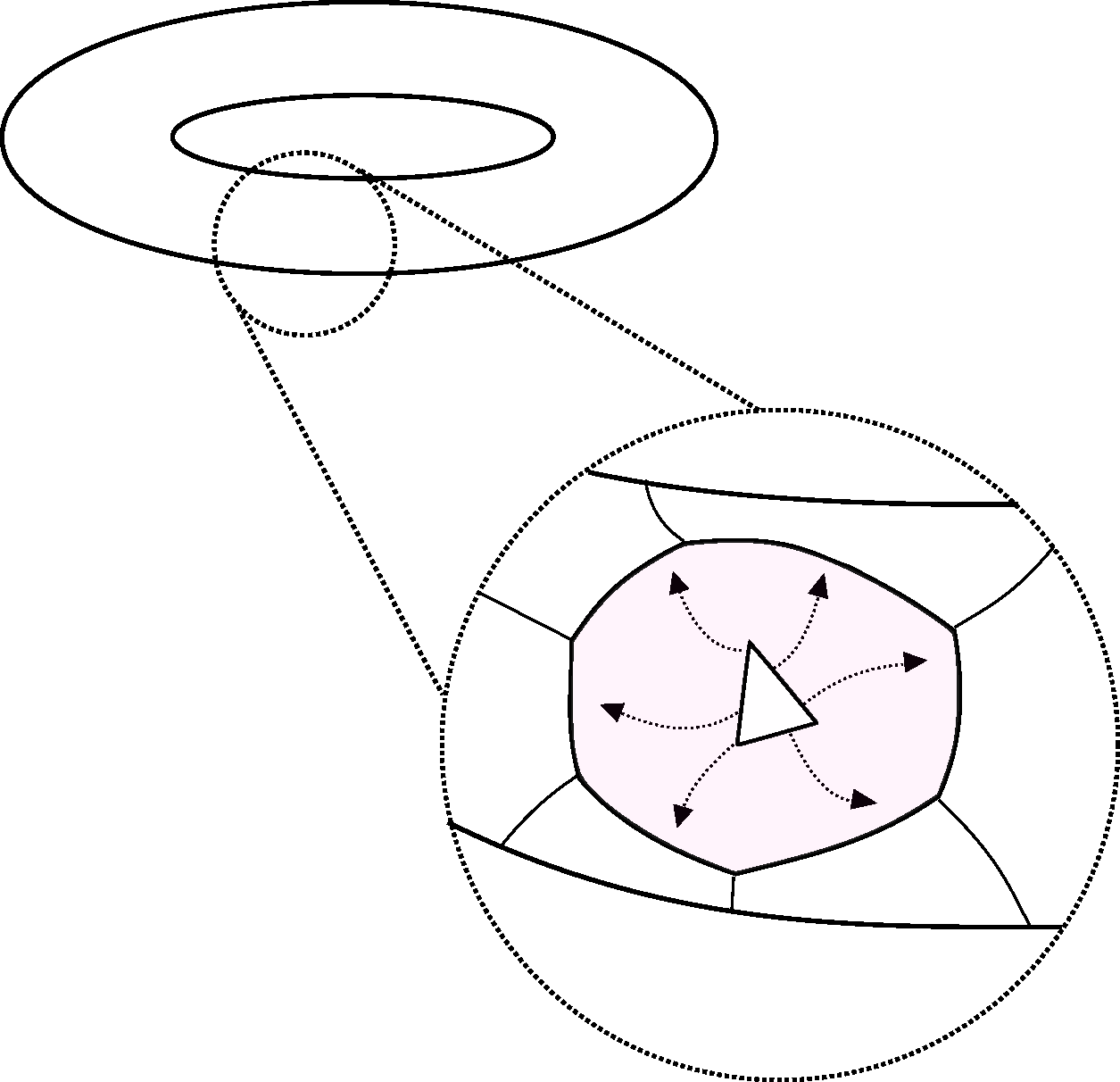}
 	\caption{\footnotesize The restriction of the triangulation to each of the $B_i$ ($i\in S$) is endo-collapsible. By counting the number of ways in which two of the $B_i$'s can be glued to one another, we determine an upper bound on the number of triangulations.}
	\label{fig:cheeger}
\end{figure}

\noindent \textbf{Part (I)}. Let $X$ be a space form of dimension $d$ satisfying the Cheeger constraints $(k,D,V)$.
By Lemma~\ref{lem:finitecover}, there exists a set $S$ of $s = s(k,D,V)$ points in such that every point in $X$ lies in distance less than $\e:=\frac{\mathcal{C}(k,D,V)}{4}$ of $S$. With this choice, any ball of radius $\e$ in $X$ is isometric to a convex ball of radius $\e$ in the unique simply-connected space form of curvature equal to the curvature of $X$. Let $T$ be a triangulation of $X$ into $N$ simplices. Let $(B_i)_{i \in \{1,\cdots,s \}}$ be the family of open convex balls with radius $\e$, centered at the points of $S$.

\smallskip

\noindent \textbf{Part (II)}. For any subset $A\subset X$, let $V_A$ denote the vertices of $\sd T$  corresponding to faces of $T$ intersecting~$A$. Define $T_{A}$ to be the 
subcomplex of $\sd T$ induced by $V_A$. 

Let now $B_i$ be one of the convex balls as above.  Then $T'_i:=  N(T_{B_i}, \sd T)$ is collapsible by Theorem~\ref{thm:ConvexEndo2}(A),  as the barycentric subdivision of $T$ can be realized by a geometric triangulation such that the vertices of $V_{B_i}$ lie in $B_i$. Also, for every face $\sigma$ of $T'_i$, $\mathrm{sd}^2 \Lk (\sigma,T'_i) $ is endo-collapsible by Proposition~\ref{prp:sphereendo}. Thus, by  Lemma~\ref{lem:endcollapse}, $\mathrm{sd}^3  T'_i$ is endo-collapsible. Hence, Theorem~\ref{thm:ENDOexpbound} provides a constant $\kappa$ such that the number of combinatorial types of $T'_i$ is bounded above by $e^{\kappa N}$.

\smallskip

\noindent \textbf{Part (III)}. 
The triangulation $\sd^2 T$ of $X$ is completely determined by
\begin{compactenum}[(i)]
\item the triangulation of each $T'_i$,
\item the triangulation $T'_i\cap T'_j$ and its position in $T'_i$ and $T'_j$. (This means we have to specify which of the faces of $T'_i$ are faces of $T'_i\cap T'_j$, too.) 
\end{compactenum}
As we saw in Part II, we have $e^{\kappa N}$ choices for triangulating each $T'_i$. Since $T'_i\cap T'_j$ is a disk, it has connected dual graph; hence, if we specify the location of one facet $\Delta$ of $T'_i\cap T'_j$ in $T'_j$ and $T'_i$ (including its orientation), this suffices to determine the position of $T'_i\cap T'_j$ in $T'_i$ and $T'_j$. For this, we have at most ${(d+1)!}^{4} N^2$ possibilities. Thus the number of geometric triangulations $T$ of $d$-dimensional space forms with $N$ facets, diameter $\le D$, volume $\ge v$, and curvature bounded below by $k$ is bounded above by
\[
{e}^{\kappa sN} \, 
{\, \left((d+1)!^2 \, N \, \right)}^{s(s-1)}. \qedhere
\]
\end{proof}
\vskip-2mm
\enlargethispage{2mm}
{\small

\subsection*{Acknowledgments}
Karim~Adiprasito acknowledges support by a Minerva fellowship of the Max Planck Society, an NSF Grant DMS 1128155, an ISF Grant 1050/16 and ERC StG 716424 - CASe.
Bruno~Benedetti acknowledges support by an NSF Grant 1600741, the DFG Collaborative Research Center TRR109, and the Swedish Research Council VR 2011-980. 
Part of this work was supported by NSF under grant No. DMS-1440140, while the authors were in residence at the Mathematical Sciences Research Institute in Berkeley, California, in Fall 2017. Both authors are thankful to the anonymous referees and to G\"unter Ziegler for corrections, suggestions, and improvements.
}

{\small

}

\end{document}